\documentclass{article}
\title{{Estimates for capacity and discrepancy of convex surfaces in sieve-like domains with an application to homogenization}}
\author{Aram L. Karakhanyan \& Martin Str\"omqvist}

\AtEndDocument{
\noindent School of Mathematics, University of Edinburgh, Edinburgh, Scotland, UK

\noindent 
Department of Mathematics, Kungliga Tekniska H\"ogskolan, Stockholm, Sweden}

\usepackage[T1]{fontenc}
\usepackage[english]{babel}
\usepackage{amssymb}
\usepackage{geometry}
\usepackage{amsmath}
\usepackage{amsrefs}
\usepackage{amssymb}
\usepackage{amsthm}
\usepackage[colorinlistoftodos, dvistyle, textsize=tiny]{todonotes} 
\usepackage{esint}
 \usepackage{relsize}

\renewcommand\cal{\mathcal}

\newtheorem{theorem}{Theorem}
\newtheorem{lemma}{Lemma}
\newtheorem{defn}{Definition}

\newcommand{\modulo}[1]{\quad(\text{mod }#1)}
\newcommand{\Ga}{\Gamma}
\newcommand{\e}{\varepsilon}
\newcommand{\Om}{\Omega}
\newcommand{\Z}{\mathbb{Z}}
\newcommand{\R}{\mathbb{R}}

\newcommand{\supp}{\operatorname{supp}}
\newcommand{\loc}{\operatorname{loc}}
\newcommand{\cp}{\operatorname{cap}}
\newcommand{\weak}{\rightharpoonup}
\renewcommand\L
{\mathlarger{\mathlarger{\mathlarger{\mathlarger{\llcorner}}}}}
\begin{document}
\maketitle

\begin{abstract}
We consider the intersection of a convex surface $\Ga$ with a periodic perforation of $\R^d$, which looks like a sieve,  given by 
$T_\e = \bigcup_{k\in \Z^d}\{\e k+a_\e T\}$ where $T$ is a given compact set and $a_\e\ll \e$ is the 
size of the perforation in the $\e$-cell $(0, \e)^d\subset \R^d$. When $\e$ 
tends to zero we establish uniform estimates for $p$-capacity $1<p<d$ and discrepancy of 
distributions of intersection $\Ga\cap T_\e$. As an application one gets that the thin obstacle problem with the obstacle defined on 
the intersection of $\Ga$ and perforations, in given bounded domain, is homogenizable when $p<1+\frac d4$.
This result is new even for the classical Laplace operator.  

\medskip 
\vspace{0.2cm}



\end{abstract}

%
%
\section{Introduction}
In this paper we study the properties of the intersection of a convex surface $\Ga$ with a periodic perforation of $\R^d$ given by 
$T_\e = \bigcup_{k\in \Z^d}\{\e k+a_\e T\}$, where $T$ is a given compact set and $a_\e$ is the 
size of the perforation in the $\e$-cell. Our primary interest is to obtain good control of $p$-capacity $1<p<d$ and discrepancy of 
distributions of the components of the intersection $\Ga\cap T_\e$ in terms of $\e$ when the size of 
perforations tends to zero. As an application of our analysis we get that the thin obstacle problem 
in periodically 
perforated domain $\Omega\subset \R^d$ with given strictly convex  and $C^2$ smooth surface as the obstacle and  $p-$Laplacian as the 
governing partial differential equation  is  
 homgenizable  provided that  $p<1+\frac d4$. Moreover, the 
limit problem 
admits a variational formulation with one extra term involving the mean capacity, see Theorem \ref{mainthm}. 
The configuration of $\Gamma$, $\Ga_\e$, $T_\e$ and $\Omega$ is illustrated in Figure 1.

This result is new even for the classical case $p=2$ corresponding to the Laplace operator.
Another novelty is contained in the proof of Theorem 2 where we use a version of the method of quasi-uniform continuity
developed in \cite{KS}. 

\newpage

\begin{figure}
\centering
\includegraphics[width=0.7\textwidth,natwidth=610,natheight=642]{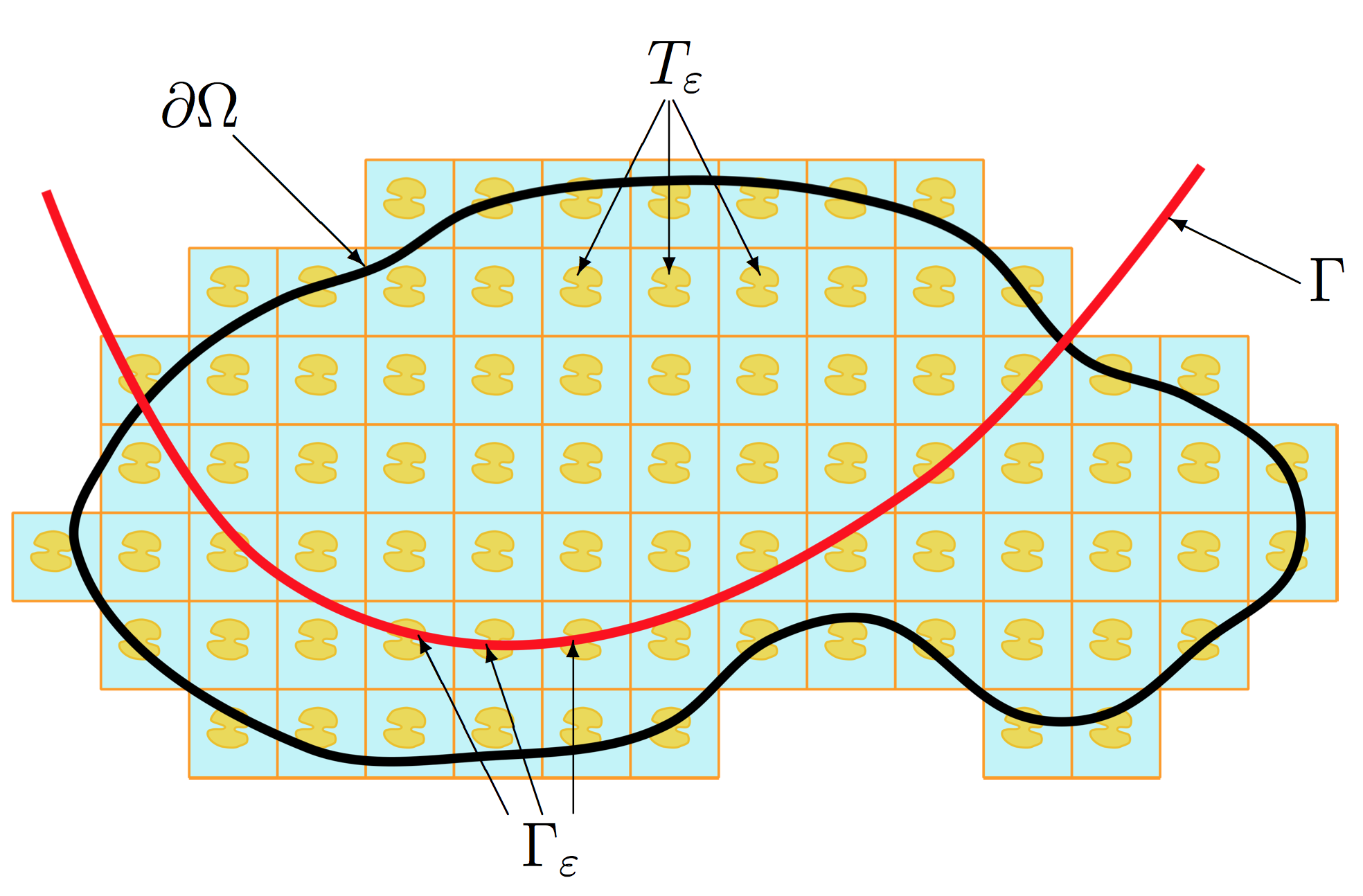}
\caption{The sieve-like configuration with convex $\Gamma$.}
\end{figure}

\subsection{Statement of the Problem} 
Let 
\[
T_\e = \bigcup_{k\in \Z^d}\{\e k+a_\e T\}, 
\] 
and let 
\[\Gamma_\e = \Gamma\cap T_\e.\]
We assume that $\Gamma$ is a strictly convex surface in $\R^d$ that locally admits the representation 
\begin{equation}\label{g}\{(x',g(x')):x'\in Q'\},\end{equation}
where $Q'\subset\R^{d-1}$ is a cube. For example, $\Gamma$ may be a compact convex surface, or may be defined globally as a graph
of a convex function. 

Without loss of generality we assume that $x_d = g(x')$ because the interchanging of coordinates preserves the structure of the periodic lattice in the definition of $T_\e$. 
We will also study homogenization of the thin obstacle problem for the $p-$Laplacian 
with an obstacle defined on $\Gamma_\e$. 
Our goal is to determine the asymptotic behaviour, as $\e\to0$, of the problem 
\begin{equation}\label{maineqe}
\min\left\{\int_\Om|\nabla v|^pdx+\int_\Om hvdx: v\in W^ {1,p}_0(\Om)\text{ and }v\ge \phi\text{ on }\Gamma_\e\right\}, 
\end{equation}
for given $h\in L^q(\Om)$, $1/p+1/q=1$ and $\phi\in W^{1,p}_0(\Om)\cap L^\infty(\Om)$. 

\medskip 

We make the following assumptions on $\Omega$, $T$, $\Gamma$, $d$ and $p$: 
\begin{itemize}
\item[$(A_1)$] $\Om\subset \R^d$ is a Lipschitz domain.

\item[$(A_2)$]
The compact set $T$ from which the holes are constructed must be sufficiently regular
in order for the mapping
\[
t\mapsto \cp(\{\Ga+t e\}\cap T)
\]
to be continuous, where $e$ is any unit vector. This is satisfied if, for example, $T$ has Lipschitz  boundary.

\item[$(A_3)$]
The size of the holes is
\[
a_\e=\e^{d/(d-p+1)}.
\]
This is the critical size that gives rise to an interesting effective equation for
\eqref{maineqe}.

\item[$(A_4)$]
The exponent $p$ in \eqref{maineqe} is in the range
\[
1<p<\frac{d+4}{4}.
\]

This is to ensure that the holes are large enough that we are able to effectively estimate the intersections
between the surface $\Ga$ and the holes $T_\e$, of size $a_\e$. See the discussion following the estimate \eqref{eq15}. 
In particular, if $p=2$ then $d>4$. 

\end{itemize}

These are the assumptions required for using the framework from \cite{KS}, though the $(A_4)$ is stricter here. 

\subsection{Main results}
The following theorems contain the main results of the present paper.


\begin{theorem}\label{thm1}
Suppose $\Gamma$ is a $C^2$ convex surface. Let $I_\e\subset [0,1)$ be an interval, let $Q'\subset \R^{d-1}$ be a cube and let 
\[A_\e = \#\left\{k'\in \Z^{n-1}\cap \e^{-1}Q': \frac{g(\e k')}{\e}\in I_\e \modulo{1}\right\}. 
\] 
Then 
\[
\left|\frac{A_\e}{N_\e}-|I_\e|\right| = O(\e^{\frac13}),
\]
where $N_\e = \#\{k'\in \Z^{d-1}\cap \e^{-1}Q'\}$. 
\end{theorem}
Next we establish an important approximation result. We use the notation $T_\e^k = \e k+a_\e T$ and $\Gamma_\e^k=\Gamma\cap T_\e^k$.
\begin{theorem}\label{thm2}
Suppose $\Gamma$ is a $C^2$ convex surface and $P_x$ a support plane of $\Gamma$ at the point $x\in \Ga$. 
Then 
\begin{itemize}
%
%
\item[{$\bf 1^\circ$}]
the $p-$capacity of $P_x^k=P_x\cap T^k_\e$ approximates $\cp_p(\Gamma^k_\e)$ as follows
\begin{equation}\label{star1}
\cp_p(\Gamma_\e^k) = \cp_p(P_{x}^k\cap \{a_\e T+\e k\}) + o(a_\e^{d-p}),  
\end{equation}
where $x\in \Gamma_\e^k$. 
%
%
\item[{$\bf 2^\circ$}]
 Furthermore, if $P_1$ and $P_2$ are two planes that intersect $\{a_\e T+\e k\}$ at a point $x$, with normals 
$\nu_1,\nu_2$ satisfying $|\nu_1-\nu_2|\le \delta$ for some small $\delta>0$, then 
\begin{equation}\label{star2}
|\cp_p(P_1\cap \{a_\e T+\e k\})-\cp_p(P_2\cap \{a_\e T+\e k\})|\le c_\delta a_\e^{d-p}, 
\end{equation}
where $\lim_{\delta\to0}c_\delta=0$.  
\end{itemize}
\end{theorem}

\medskip 
As an application of Theorems 1-2 we have 

\begin{theorem}\label{mainthm}
Let $u_\e$ be the solution of \eqref{maineqe}. Then $u_\e\weak u$ in $W_0^{1,p}(\Om)$ as $\e\to0$, 
where $u$ is the solution to 
\begin{equation}\label{maineqhom}
\min\left\{\int_\Om|\nabla v|^pdx+ \int_{\Gamma\cap\Om}|(\phi-v)_+|^p\cp_{p,\nu(x)}(T)d\cal{H}^{d-1}+\int_\Om fvdx: v\in W^ {1,p}_0(\Om)\right\}.
\end{equation}
\end{theorem}
In \eqref{maineqhom}, $\nu(x)$ is the normal of $\Gamma$ at $x\in \Gamma$ and $\cp_{p,\nu(x)}(T)$ is the mean $p$-capacity of $T$ with respect to the 
hyperplane $P_{\nu(x)}=\{y\in \R^ d:\nu(x)\cdot y=0\}$, given by 
\begin{equation}\label{meancap}
\cp_{p,\nu(x)}(T) = \int_{-\infty}^ {\infty}\cp_p(T\cap\{P_{\nu(x)}+t\nu(x)\})dt, 
\end{equation}
where $\cp_p(E)$ denotes $p$-capacity of $E$ with respect to $\R^ d$. 

Theorem \ref{mainthm} was proved by the authors in \cite{KS} under the assumption that $\Gamma$ is a hyper plane, which was in turn a generalization of the paper \cite{LSY}. In a larger context, Theorem \ref{mainthm} contributes to the theory of homogenization in non-periodic perforated domains, in that the support of the obstacle, $\Gamma_\e$, is not periodic. Another class of well-studied non-periodic 
perforated domains, not including that of the present paper, 
is the random stationary ergodic domains introduced in \cite{CaffM}. In the case of stationary ergodic domains the perforations are situated on lattice points, which is not the case for the set $\Gamma_\e$. The perforations, i.e.\ the components of $\Gamma_\e$, have desultory (though deterministic by definition) distribution. For the periodic setting \cite{CM} is a standard reference.

The proof of Theorem \ref{mainthm} has two fundamental ingredients. First the structure of the set $\Gamma_\e$ is analysed using tools from the theory of uniform distribution,  Theorem 1. 
We prove essentially that the components of $\Gamma_\e$ are uniformly distributed over $\Gamma$ with a good bound on the discrepancy. This is achieved by studying the distribution of the sequence 
\begin{equation}\label{gseq}\{\e^{-1}g(\e k')\}_{k'},\end{equation} for $g$ defined by \eqref{g} and $\e k'\in Q'$.  
Second, we construct a family of well-behaved correctors based on the result of Theorem 2.

The major difficulty that arises when $\Gamma$ is a more general surface than a hyperplane is to estimate the discrepancy of the distribution of 
(the components of) $\Gamma_\e$ over $\Gamma$, which is achieved through studying the discrepancy of $\{\e^{-1}g(\e k')\}_{k'}$. For a definition of discrepancy, see section \ref{EKthm}. In the framework of uniform convexity we can apply a theorem of Erd\"os and Koksma which gives good control of the discrepancy.

%
%
 \section{Discrepancy and the Erd\"os-Koksma theorem}\label{EKthm}

In this section we formulate a general result for the uniform distribution of a sequence and 
derive a decay estimate for the corresponding discrepancy. 

\begin{defn}\label{defn-dscrp}
The discrepancy of the first $N$ elements of a sequence $\{s_j\}_{j=1}^\infty$ is given by 
\[D_N=\sup_{I\subset(0,1]}\left|\frac{A_N}{N}-|I|\right|,\]
where $I$ is an interval, $|I|$ is the length of $I$ and $A_N$ is the number of $1\le j\le N$ for which $s_j\in I\modulo{1}$. 
\end{defn}

We first recall the Erd\"os-Tur\'an inequality, see Theorem 2.5 in \cite{KN},  
for the discrepancy of the sequence $\{s_j\}_{j=1}^\infty$ 
\begin{equation}\label{erd-tur}
D_N\leq \frac1n+\frac1N\sum_{k=1}^n \frac1k \left|\sum_{j=1}^Ne^{2\pi if(j)k}\right|
\end{equation}
where $n$ is a parameter to be chosen so that 
the right hand side has optimal decay as $N\to \infty$. 
Observe that $s_j$ is the $j-$th element of the sequence which in our case is $s_j=f(j)$ for a given 
function $f$ and $N=\left[ \frac1\e \right]$.

We employ the following estimate of Erd\"os and Koksma (\cite{KN}, Theorem 2.7) in order to estimate 
the second sum in (\ref{erd-tur}): let $a, b\in \mathbb N$ such that $0<a<b$ then one has the estimate 
\begin{equation}\label{erd-koks}
\left|\sum_{j=1}^Ne^{2\pi if(j)k}\right|\le (|F_k'(b)-F_k'(a)|+2)\left(3+\frac1{\sqrt\rho}\right)
\end{equation}
where $F_k(t) = kf(t)$ and $F_k''(t)\ge\rho>0$ for some positive number $\rho.$
In order to apply this result to our problem we first need to reduce the dimension of \eqref{gseq} to one. To do so
let us assume that the obstacle $\Gamma$ is given as the graph of a function $x_d=g(x')$ where 
$g$ is strictly convex $C^2$ function such that 
\begin{equation}\label{uni-convex}
c_0\delta_{\alpha, \beta}\le D_{x_\alpha x_\beta}g(x')\le C_0\delta_{\alpha, \beta}, \quad 1\le \alpha, \beta\le d-1
\end{equation}
for some positive constants $c_0<C_0$. 

Next we rescale the $\e$-cells and consider the normalised problem in the unit cube $[0, 1]^d$. The resulting
function is $f(j)=\frac{g(\e j)}{\e}, j\in \Z^{d-1}$. 

If $d=2$ then we can directly apply (\ref{erd-koks}) to the scaled function $f$ above.
Otherwise for $d>2$ we need an estimate for the multidimensional discrepancy in terms 
of $D_N$ introduced in Definition \ref{defn-dscrp}, a similar idea was used in   \cite{KS} for the linear obstacle.
Suppose for a moment that this is indeed the case. Then 
we can take $F_k(t)=kf(t)$ in \eqref{erd-koks} 
and noting 
\begin{equation}\label{Hessian}
D_{x_\alpha}f(x')=kD_\alpha g(\e x'), \qquad D_{x_\alpha}^2f(x')=k\e D_\alpha^2 g(\e x')\ge k\e c_0, \quad 1\le \alpha\le d-1
\end{equation}
one can 
proceed as follows

\begin{eqnarray*}
\left|\sum_{j=1}^Ne^{2\pi if(j)k}\right|&\le& (|kD_{x_\alpha}g(\e N)-kD_\alpha g(\e )|+2)\left(3+\frac1{\sqrt{k\e c_0}}\right)\\\nonumber
&\le &(k\e C_0(N-1)+2)\left(3+\frac1{\sqrt{k\e c_0}}\right)\\\nonumber
&\le &k\left(\e C_0(N-1)+\frac2k\right)\left(3+\frac1{\sqrt{k\e c_0}}\right)\\\nonumber
&\le &k\left(\e C_0(N-1)+\frac2k\right)\left(3+{\sqrt{\frac N{kc_0}}}\right)\\\nonumber
&\le&\lambda k\left(1+\sqrt{\frac N{k}}\right)
\end{eqnarray*}
for some tame constant $\lambda >0$ independent of $\e, k$. Plugging this into 
\eqref{erd-tur} yields

\begin{eqnarray*}
D_N&\le& \frac1n+\frac{\lambda}N\sum_{k=1}^n\left(1+\sqrt{\frac N{k}}\right)\\\nonumber
&=&\frac1n +\frac{\lambda n}N+ \frac{\lambda}{\sqrt N} \sum_{k=1}^n \frac1{\sqrt{k}}\\\nonumber
&\leq& \frac1n+\overline{\lambda}\sqrt{\frac nN}\left(1+\sqrt{\frac n{N}}\right)
\end{eqnarray*}
for another tame constant $\overline \lambda>0$.
Now to get the optimal decay rate we choose $\frac1n =\sqrt{\frac nN}$ which yields $N=n^3$ and hence
$$n=N^{\frac13}\approx \frac1{\e^{\frac13}}$$
and we arrive at the estimate

\begin{equation}\label{DN-est}
D_N = O( \e^{\frac13}).
\end{equation}

{\subsection{Proof Theorem 1}}
\begin{proof}
Suppose $Q' $ is a cube of size $r$. Then there is a cube $Q''\subset \R^{d-2}$ such that $Q' = [\alpha,\beta]\times Q'$, $\beta-\alpha=r$. 
We may rewrite $A_\e$ as 
\begin{align*}
A_\e&=\sum_{k''\in \e^{-1}Q''\cap\Z^{d-2}}\#\left\{k_1\in\Z:a\le k_1\le b\text{ and }\e^{-1}g(\e k_1+\e k'')\in I_\e\modulo{1}\right\},
\end{align*}
where $(k_1,k'')=k'$, $a,b$ are the integer parts of $\e^{-1}\alpha$ and $\e^{-1}\beta$ respectively and $|(b-a)- \e^{-1}r|\le 1$. We also note that $N_\e = (\e^{-1}r)^{d-1}+O(\e^{-1}r)^{d-2}$. 
Consider 
 $$A_\e^1(k'')=\#\left\{k_1\in\Z:a\le k_1\le b\text{ and }\e^{-1}g(\e k_1+\e k'')\in I_\e\modulo{1}\right\}.$$ 
 Then we have 
\begin{align}\label{decomp}
\frac{A_\e}{N_\e}-|I_\e|&=\frac{1}{(\e^{-1}r)^{d-2}}\sum_{k''\in \e^{-1}Q''\cap\Z^{d-2}}\frac{A_\e^1(k'')}{(\e^{-1}r)}-|I_\e|. 
\end{align}
For each $k''$ the function $h:s\to \e^{-1}g(\e s+\e k'')$ satisfies 
$|h'(s)|\le C_1$ and $h''(s)\ge \rho\e$ for $a\le s\le b$. Thus we may apply the Erd\"os-Koksma Theorem as described above and conclude that 
\[
\left|\frac{A_\e^1(k'')}{(\e^{-1}r)}-|I_\e|\right|\le C\e^{\frac13}.
\]
 It follows that the modulus of the left hand side of \eqref{decomp} is bounded 
by $C\e^{\frac13}$, proving the theorem. 
\end{proof}

\medskip

%
%
\section{Correctors }

The purpose of this section is to construct a sequence of correctors that satisfy the hypotheses given below. 
Once we have established the existence of these correctors, the proof of the Theorem \ref{mainthm} is identical to the planar case treated in \cite{KS}.
\begin{itemize}
\item[$\bf H1$] $0\leq w_\e\leq 1$ in $\R^d$, $w_\e=1$ on $\Gamma_\e$ and $w_\e \weak 0$ in $W^{1,p}_{\loc}(\R^d)$,

\item[$\bf H2$] $\int_\Om |\nabla w_\e|^pfdx\to \int\limits_{\Gamma} f(x)\cp_{p,\nu_x}d\cal H^{d-1}$, for any $f\in W_0^{1,p}(\Omega)\cap L^\infty(\Omega)$,

\item[$\bf H3$] (weak continuity) for any $\phi_\e\in W_0^{1,p}(\Om)\cap L^\infty(\Om)$
such that
\[
\left\{\begin{array}{l}
\sup\limits_{\e>0}\|\phi_\e\|_{L^\infty(\Om)}<\infty,\\
\phi_\e=0 \text{ on } \Ga_\e \text{ and } \phi_\e\weak \phi\in W_0^{1,p}(\Om),
\end{array}\right.
\]
we have
\[
\langle-\Delta_p w_\e, \phi_\e\rangle\rightarrow
\langle\mu, \phi \rangle
\]
with
\begin{equation}\label{mu-0}
 d\mu(x)=\cp_{p,\nu(x)}d\cal H^{d-1}\L  \Ga,
\end{equation}
where $
\cp_{p,\nu(x)}$ is given by \eqref{meancap} and $\cal H^s\L\Ga$ is the restriction of $s-$dimensional
Hausdorff measure on $\Ga$.
\end{itemize}

Setting  $\Gamma_\e^k:=\Gamma\cap \{a_\e T+\e k\}\neq\emptyset$, 
we define $w_\e^k$ by 
\begin{align}\nonumber
&\Delta_p w_\e^k=0\text{ in }B_{\e/ 2}(\e k)\setminus \Gamma_\e^k,\\\nonumber
&w_\e^k=0 \text{ on } \partial B_{\e/ 2}(\e k),\\\nonumber
&w_\e^k = 1\text{ on }\Gamma_\e^k. 
\end{align}

Then it follows from the definition of $\cp_p$  \cite{Frehse} that 

\[
\int_{B_{\e/2}(\e k)}|\nabla w_\e^k|^pdx = \cp_p(\Gamma_\e^k) + o(a_\e^{d-p}). 
\]
Indeed, we have 
\begin{eqnarray*}
\cp_p(\Gamma_\e^k, B_{\e/2}(\e k) )&=&\inf\left\{\int_{B_{\e/2}}|\nabla w|^p : w\in W^{1,p}_0(B_{\e/2}(\e k)) \ {\rm and} \ w=1 \ {\rm on } \ \Gamma_\e^k \right\}\\\nonumber
&=&a_\e^{d-p}\inf\left\{\int_{B_{\e/{2a_\e}}}|\nabla w|^p : w\in W^{1,p}_0(B_{\e/2a_\e} \ {\rm and} \ w=1 \ {\rm on } \ \frac1{a_\e}\Gamma_\e^k \right\}\\\nonumber
&=&a_\e^{d-p}\left(\cp_p\left(\frac1{a_\e}\Gamma_\e^k\right)+o(1)\right)\\\nonumber
&=&\cp_p\left(\Gamma_\e^k\right)+o(a_\e^{d-p}).
\end{eqnarray*}

Note that $\cp_p(\Gamma_\e^k)=O(a_\e^{d-p})$ since $\Gamma_\e^k=\Gamma\cap\{\e k+ a_\e T\}$ and $\cp_p(t E)=t^{d-p}\cp_p(E)$ if $t\in \R_+$ and $E\subset \R^d$. 
If $Q'$ is a cube in $\R^{d-1}$, the components of $\Gamma_\e\cap Q'\times\R$ are of the form 
$\Gamma_\e^k=\Gamma\cap \{(\e k',\e k_d)+a_\e T\}$ for $\e k'\in Q'$. In particular, 
$\Gamma_\e^ k\neq\emptyset $ if and only if $\e^{-1} g(\e k')\in I_\e \modulo{1}$ where $|I_\e|=O(a_\e/\e)$.   
Thus Theorem \ref{thm1} tells us that the number of components of $\Gamma_\e\cap Q'\times\R$ equals 
$A_\e = |I_\e|N_\e+N_\e O(\e^{\frac13})$, or explicitly 
\begin{equation}\label{eq15}
\left|\frac{\frac{A_\e}{N_\e}}{\frac{a_\e}{\e}}-1\right|=\frac{{O}(\e^{\frac13})}{\frac{a_\e}\e}.
\end{equation}
Here we need to have $\e^{1/3}=o(|I_\e|)$, 
which is equivalent to $(A_4)$. Since 
\[\int_{B_{\e/2}(\e k)}|\nabla w_\e^k|^pdx = \cp_p(\Gamma_\e^k) + o(a_\e^{d-p}),\] 
we get 
\[
\int_{\R\times Q'}|\nabla w_\e|^ pdx\le C(|I_\e|N_\e\cp_p(\Gamma_\e^k))\le C\frac{a_\e}{\e}\e^ {1-d}|Q'|a_\e^ {n-p}=C|Q'|. 
\]
Thus $\int_K|\nabla w_\e|^p$ is uniformly bounded on compact sets $K$. Since $w_\e(x)\to 0$ pointwise for $x\not\in \Gamma$, 
$\bf H1$ follows. 

When verifying $\bf H_2$ and $\bf H_3$ we will only prove that 
\begin{equation}\label{key-00}
\lim_{\e\to0}\int_Q|\nabla w_\e|^pdx= \int_{\Gamma\cap Q}c_{\nu(x)}d\mathcal{H}^{d-1}(x), \quad \text{for all cubes }Q\subset \R^d. 
\end{equation}
Once this has been established the rest of the proof is identical to that given in \cite{KS}.

\medskip

\section{Proof of Theorem 2}

\begin{proof}
{$\bf 1^\circ$}\ 
Set  $R_\e=\frac\e{2a_\e}\to \infty$, then after scaling 
we have to prove that 
\begin{equation}\label{o-small0}
\int_{B_{R_\e}}|\nabla v_1|^p-\int_{B_{R_\e}}|\nabla v_2|^p=o(1)
\end{equation}
uniformly in $\e$ where 
\begin{align}\nonumber
&\Delta_p v_i=0\text{ in } B_{R_\e}\setminus S_i,\\\nonumber
&v_i=0 \text{ on } \partial B_{R_\e},\\\nonumber
&v_i = 1\text{ on }S_i. 
\end{align}
and $S_1=\frac1{a_\e}\Gamma^k_\e, S_2=\frac1{a_\e}P_{x}$.

We approximate  $v_i$ in the domain $B_{R_\e}\setminus D^t_i$ with  
$D^t_i$ being a bounded domain with smooth boundary and $D^t_i\to S_i $ as $t\to 0$ in Hausdorff distance. 
Consider 
\begin{align}\nonumber
&\Delta_p v_i^t=0\text{ in } B_{R_\e}\setminus D^t_i, \\\nonumber
&v_i^t=0 \text{ on } \partial B_{R_\e},\\\nonumber
&v_i^t = 1\text{ on }\partial D^t_i. 
\end{align}

Observe  that $\int_{B_{R_\e}\setminus D_i^t}|\nabla v_i^t|^{p}, i=1, 2$ remain bounded 
as $t\to 0$ thanks to Caccioppoli's inequality. 
Indeed, 
$w=(1-v_i^t)\eta\in W^{1, p}_0(B_5\setminus D_i^t)$ where $\eta\in C_0^\infty(B_5)$ such that 
$0\leq \eta\leq 1$ and $\eta\equiv1$ in $B_3$. Using  $w$ as a test function we conclude that 
$$\int_{B_5\setminus D_i^t}|\nabla v_i^t|^p\eta=\int_{B_5\setminus D_i^t}|\nabla v_i^t|^{p-2}\nabla v_i^t\nabla\eta (1-v_i^t).$$  
Since  $\eta\equiv 1$ in $B_3$ then applying H\"older inequality we infer that 
$\int_{B_3\setminus D_i^t}|\nabla v_i^t|^p\le C\int_{B_5}(1-v_i^t)^p$.
In $B_{R_\e}\setminus B_2$ the $L^p$ we compare $W(x)=|x/2|^{\frac{p-d}{p-1}}$ with $v_i$. Note that 
our assumption $A_4$ implies that $p<d$. Moreover, since $W$ is $p-$harmonic in 
$B_{R_\e}\setminus B_2$ then the comparison principle yields $v_i\leq W$ in $B_{R_\e}\setminus B_2$.
From the proof of Caccioppoli's inequality above choosing non-negative $\eta\in C^\infty(\R^d)$ such that $\eta\equiv 0$ in $B_2$, $\frac12\le \eta\le 1$ in $B_{R_\e}\setminus B_3$,  
and $\eta=1$ in $\R^d\setminus B_{R_\e}$ and using $\eta v_i\in W_0^{1, p}(B_{R_\e}\setminus B_2)$ 
as a test function we infer 
$$\int_{B_{R_\e}\setminus B_3}|\nabla v_i|^p\leq \frac C{R^p_\e}\int_{B_{R_\e}\setminus B_2}v_i^p\leq \frac{C}{R_\e^{\frac1{p-1}}}\to 0 \quad\text{as}\ \e\to 0$$
 where the last 
bound follows from the estimate $v_i\leq W$. Combining these estimates we infer
\begin{equation}\label{Cacc}
\|v_i^t\|_{W^{1, p}(B_{R_\e})}\le K, \quad i=1,2
\end{equation}
for some tame constant $K$ independent of $t$ and $\e$.
Thus, by construction $v^t_i\rightharpoonup v_i$ weakly in $W^{1, p}_0(B_{R_\e})$.

\medskip 

Let $\psi\in C^\infty(\R^d)$ such that $\supp \psi\supset D_1^t\cup D_2^t$ 
and $\psi\equiv 1$ in $\R^d\setminus B_2$. Then the function $\psi(v_1^t-v_2^t)\in W^{1,p}_0(B_{R_\e})$ and  
it vanishes on $\supp\psi\supset D_1^t\cup D_2^t$. 
Thus we have 
\begin{eqnarray*}
\int_{B_{R_\e}}(\nabla v_1^t|\nabla v_1^t|^{p-2}-\nabla v_2^t|\nabla v_2^t|^{p-2})(\nabla v_1^t-\nabla v_2^t)\psi=-\int_{B_{R_\e}}(\nabla v_1^t|\nabla v_1^t|^{p-2}-\nabla v_2^t||\nabla v_2^t|^{p-2})(v_1^t- v_2^t)\nabla \psi
\end{eqnarray*}
Note that $v_1^t-v_2^t=0$ on $D^t_1\cap D_2^t$.
Choosing  a sequence $\psi_n$ such that $1-\psi_m$ converges to the characteristic function 
$\chi_{D_1^t\cup D_2^t}$ of the set $D_1^t\cup D_2^t$ we conclude 

\begin{equation}\label{26}
\int_{B_{R_\e}}(\nabla v_1^t|\nabla v_1^t|^{p-2}-\nabla v_2^t|\nabla v_2^t|^{p-2})(\nabla v_1^t-\nabla v_2^t)=
J_1+J_2
\end{equation}
where 
\begin{eqnarray*}
J_1=\int_{\partial D_1^t} (1- v_2^t)[\partial_\nu v_1^t|\nabla v_1^t|^{p-2}-\partial_\nu v_2^t|\nabla v_2^t|^{p-2}], \\\nonumber
J_2=\int_{\partial D_2^t}(v_1^t- 1)[\partial_\nu v_1^t|\nabla v_1^t|^{p-2}-\partial_\nu v_2^t|\nabla v_2^t|^{p-2}].
\end{eqnarray*}

Notice that on $\partial D^t_i$ we have that $\nu=-\frac{\nabla \psi_m}{|\nabla \psi_m|}$ is the unit normal pointing inside 
$D^t_i$. We denote $n=-\nu$ and then we have that

\begin{eqnarray*}
-\int_{\partial D_1^t} (1- v_2^t)\partial_\nu v_2^t|\nabla v_2^t|^{p-2}&=&\int_{\partial D_1^t} (1- v_2^t)\partial_n v_2^t|\nabla v_2^t|^{p-2}\\\nonumber
&=&\int_{\partial (D_1^t\cap D^t_2)} (1- v_2^t)\partial_n v_2^t|\nabla v_2^t|^{p-2}\\\nonumber
&=&\int_{D_1^t\setminus D_2^t} \text{div}((1-v_2^t)\nabla v_2^t|\nabla v_2^t|^{p-2})\\\nonumber
&=&-\int_{D_1^t\setminus D_2^t} |\nabla v_2^t|^p,
\end{eqnarray*}
and similarly 
\begin{eqnarray*}
\int_{\partial D_2^t}(v_1^t- 1)\partial_\nu v_1^t|\nabla v_1^t|^{p-2}=-\int_{D_2^t\setminus D_1^t} |\nabla v_1^t|^p .
\end{eqnarray*}
Setting 
\begin{equation}\label{exprs-I}
I=\int_{B_{R_\e}}(\nabla v_1^t|\nabla v_1^t|^{p-2}-\nabla v_2^t|\nabla v_2^t|^{p-2})(\nabla v_1^t-\nabla v_2^t)
\end{equation}
and returning to \eqref{26} we infer 

\begin{eqnarray}\nonumber
I&=&-\int_{D_1^t\setminus D_2^t} |\nabla v_2^t|^p-\int_{D_2^t\setminus D_1^t} |\nabla v_1^t|^p+
\int_{\partial D_1^t} (1- v_2^t)\partial_\nu v_1^t|\nabla v_1^t|^{p-2}
- \int_{\partial D_2^t}(v_1^t- 1)\partial_\nu v_2^t|\nabla v_2^t|^{p-2}\\\nonumber
&\le &\int_{\partial D_1^t} (1- v_2^t)\partial_\nu v_1^t|\nabla v_1^t|^{p-2}
- \int_{\partial D_2^t}(v_1^t- 1)\partial_\nu v_2^t|\nabla v_2^t|^{p-2}\\\nonumber
&\le &\sup_ {D_1^t}(1-v_2^t)\int_{\partial D_1^t} |\partial_\nu v_1^t||\nabla v_1^t|^{p-2}
+\sup_ {D_2^t}(1-v_1^t) \int_{\partial D_2^t}|\partial_\nu v_2^t||\nabla v_2^t|^{p-2}.\\\nonumber
\end{eqnarray}
But on $\partial D_i^t$ we have  $\partial_\nu v_i^t\ge 0$ ($\nu$ points inside $D_i^t$)
 because $v_i^t$ attains its maximum on $\partial D_i^t$.
Thus we can omit the absolute values of the normal derivatives and obtain 
\begin{eqnarray*}
I&\le& \sup_ {D_1^t}(1-v_2^t)\int_{\partial D_1^t} \partial_\nu v_1^t|\nabla v_1^t|^{p-2}
+\sup_ {D_2^t}(1-v_1^t) \int_{\partial D_2^t}\partial_\nu v_2^t|\nabla v_2^t|^{p-2}\\\nonumber
&=&
\sup_ {D_1^t}(1-v_2^t)\int_{B_{R_\e}\setminus D_1^t} \text{div}(v_1\nabla v_1^t|\nabla v_1^t|^{p-2})
+\sup_ {D_2^t}(1-v_1^t) \int_{B_{R_\e}\setminus D_2^t}\text{div}(v_2\nabla v_2^t|\nabla v_2^t|^{p-2})\\\nonumber
&=&\sup_ {D_1^t}(1-v_2^t)\int_{B_{R_\e}\setminus D_1^t}|\nabla v_1^t|^{p}
+\sup_ {D_2^t}(1-v_1^t) \int_{B_{R_\e}\setminus D_2^t}|\nabla v_2^t|^{p}.
\end{eqnarray*}
Recall that by Lemma 5.7 \cite{MZ} there is a generic constant $M>0$ such that 
\begin{equation}\label{coercive}
(|\xi|^{p-2}\xi-|\eta|^{p-2}\xi)(\xi-\eta)\ge M\left\{
\begin{array}{lll}
|\xi-\eta|^p &\text{if}\ p>2,\\
|\xi-\eta|^2(|\xi|+|\eta|)^{p-2} &\text{if}\ 1<p\le 2
\end{array}
\right.
\end{equation}
for all $\xi, \eta\in \R^d$.

First suppose that $p>2$ then applying  inequality \eqref{coercive} to \eqref{exprs-I} yields
\begin{eqnarray*}
I \ge
M\int_{B_{R_\e}}|\nabla v_1^t-\nabla v_2^t|^{p}.
\end{eqnarray*}
As for the case $1<p\le 2$ then from \eqref{coercive} we have 

\begin{eqnarray*}
I\ge
M\int_{B_{R_\e}}|\nabla v_1^t-\nabla v_2^t|^2(|\nabla v_1^t|+|\nabla v_2^t|)^{p-2}.
\end{eqnarray*}
But, from H\"older's inequality and \eqref{Cacc} we get 
\begin{eqnarray}
\int_{B_{R_\e}}|\nabla v_1^t-\nabla v_2^t|^p&=& \int_{B_{R_\e}}|\nabla v_1^t-\nabla v_2^t|^p(|\nabla v_1^t|+|\nabla v_2^t|)^{\frac{p(p-2)}2}(|\nabla v_1^t|+|\nabla v_2^t|)^{-\frac{p(p-2)}2}\\\nonumber
&\le&\left( \int_{B_{R_\e}}|\nabla v_1^t-\nabla v_2^t|^2(|\nabla v_1^t|+|\nabla v_2^t|)^{p-2}\right)^{\frac p2}
\left(\int_{B_{R_\e}}(|\nabla v_1^t|+|\nabla v_2^t|)^p\right)^{1-\frac p2}\\\nonumber
&\le&\left(\frac{I}M\right)^{\frac p2} (2K)^{1-\frac p2}.
\end{eqnarray}
Therefore, there is a tame constant $M_0$ such that for any $p>1$ we have 
\begin{equation*}
\int_{B_{R_\e}}|\nabla v_1^t-\nabla v_2^t|^{p}\le M_0\left[\sup_ {D_1^t}(1-v_2^t)\int_{B_{R_\e}\setminus D_1^t}|\nabla v_1^t|^{p}
+\sup_ {D_2^t}(1-v_1^t) \int_{B_{R_\e}\setminus D_2^t}|\nabla v_2^t|^{p}
\right]^{\min(1, \frac p2)}.
\end{equation*}
Letting $t\to 0$ we get 
\begin{eqnarray}\label{sstt}
\int_{B_{R_\e}}|\nabla v_1-\nabla v_2|^{p}&\le& \liminf_{t\to 0} \int_{B_{R_\e}}|\nabla v_1^t-\nabla v_2^t|^{p}\\\nonumber
&\le& M_1\liminf_{t\to 0}\left[\sup_ {D_1^t}(1-v_2^t)
+\sup_ {D_2^t}(1-v_1^t) \right]^{\min(1, \frac p2)}.
\end{eqnarray}
with some tame constant $M_1$. 


Since $1-v_i^t$ are nonnegative $p-$subsolutions in $B_{R_\e}$,  from the weak maximum 
principle, Theorem 3.9 \cite{MZ} we obtain 
\begin{equation}\label{w-max}
\sup_{B_{\sigma r}(z)}(1-v_i^t)\leq \frac C{(1-\sigma)^{n/p}}\left(\fint_{B_{r}(z)} (1-v_i^t)^p\right)^{\frac1p}.
\end{equation}

Take a finite covering of $D^t_i$ with balls $B_r(z_k^i), z_k^i\in S_i, r=3a_\e, k=1, \dots, N$. 
Choose $t$ small enough such that $D^t_j\subset \bigcup_{k=1}^N B_r(z_k^i)$ and 
applying \eqref{w-max} we obtain for $ i,j\in\{1, 2\}$ with $i\not=j$
\begin{eqnarray*}
\sup_ {D_j^t}(1-v_i^t)\leq \max_k\sup_{B_{ r}(z_k^i)}(1-v_i^t)\leq  C\max_k\left(\fint_{B_{2r}(z_k^i)} (1-v_i^t)^p\right)^{\frac1p}.
\end{eqnarray*}

Since  $\|v_i^t\|_{W^{1,p}(B_3)}\le C$ uniformly 
for all $t>0$ it follows that  $v_i^t\to v_1$ strongly in 
$L^p(B_3)$ and $v_i$ is  quasi-continuous. In other words, for any positive number $\theta$ there is a set $E_{\theta}$ such that 
$\cp_p E_\theta<\theta$ and $v_i$ is  continuous in $B_2\setminus E_\theta$.  Notice that 
$E_\theta\subset S_1\cup S_2$ and hence $\mathcal H^d(E_\theta)=0$.

This yields
\begin{eqnarray}
\lim_{t\to 0} \fint_{B_{r}(z_k^i)} (1-v_i^t)^p=\fint_{B_{r}(z_k^i)} (1-v_i)^p&=&\fint_{B_{2r}(z_k^i)\cap E_\theta} (1-v_i)^p+\fint_{B_{2r}(z_k^i)\setminus E_\theta} (1-v_i)^p\\\nonumber
&=&\fint_{B_{2r}(z_k^i)\setminus E_\theta} (1-v_i)^p\\\nonumber
&\le& C[\omega_{i}(6a_\e)]^p
\end{eqnarray}
where $\omega_i(\cdot)$ is the modulus of continuity of $v_i$ on $B_{3}$ modulo the set $E_\theta$.
Thus 
$$\int_{B_{R_\e}}|\nabla v_1-\nabla v_2|^{p}\le C[\omega_{1}(6a_\e)+\omega_{2}(6a_\e)]^{p\min(1, \frac p2)}.$$
Hence \eqref{o-small0} is established.
Rescaling back and noting that $a_\e^{d-p}\omega_i(a_\e)=o(a_\e^{d-p})$ the result follows.
Observe that $L^p$ norm of $\nabla v_i^t$ remains uniformly bounded in $B_{R_\e}$ by \eqref{Cacc} 
and hence the moduli of quasi-continuity in, say,  $B_3$ do not depend on the particular choice of $\Gamma_\e^k$ or the 
tangent plane $P_{x}^k$.

\medskip 

${\bf 2^\circ}$
We recast the argument above but now for $S_1=\frac1{a_\e}P_1, S_2=\frac1{a_\e}P_{2}$.
Squaring the inequality $|\nu_1-\nu_2|\leq \delta$ we get that 
$2\sin\frac\beta2\le \delta$ where $\beta$ is the angle between $P_1$
and $P_2$. Since $\delta$  now measures the deviation of $v_1^t$ from 1 on $D_2^t$, (resp. $v_2^t$ on $D_1^t$)
we conclude that the corresponding moduli of continuity of the limits $v_1, v_2$ (as $t\to 0$) modulo a
set $E_\theta\subset S_1\cup S_2$ with small $p-$capacity depend on $\delta$, i.e. 
\begin{eqnarray}
\fint_{B_{r}(z_k^i)} (1-v_i)^p\le C[\omega_{i}(12\delta)]^p
\end{eqnarray}
where $B_r(z_k^i)$ provide a covering of $D_i^t$ as above but now, say,  $r=6\delta$.
Hence we can take $c_\delta=C(\omega_1(12\delta)+\omega_2(12\delta)).$
 \end{proof}

\medskip

\section{Proof of Theorem \ref{mainthm}}
We now formulate our result on the local approximation of total capacity (say in $Q'$) by tangent planes of $\Ga$
and prove \eqref{key-00}.

\begin{lemma}\label{lemma3}
Fix a cube $Q'\subset \R^{d-1}$ such that if $x=(x',x_d)$ and $y=(y',y_d)$ belong to $\Gamma$ and $x',y'\in Q'$, then the normals $\nu_x,\nu_y$ of $\Gamma$ at $x$ and $y$ satisfy 
$|\nu_x-\nu_y|\le \delta $.  
Then for any $x=(x',x_d)\in\Gamma$ with $x'\in Q'$, there holds 
\[
\lim_{\e\to0}\sum_{k\in \Z^n:k'\in \e^{-1}Q'}\int_{B_\e^k}|\nabla w_\e^k|^pdx = [\cp_{p,\nu_x}(T) + O(C_\delta)]\mathcal{H}^{d-1}(\Gamma_{Q'}), 
\]
where $\lim_{\delta\to0}C_\delta=0$ and $\Gamma_{Q'}=\{x\in\Gamma:x'\in Q'\}$.

\end{lemma}

\begin{proof}
Fix $x\in \Gamma_{Q'}$ and let $P$ be the plane $\{y:y\cdot \nu_x=0\}$, where $\nu_x$ is the normal of $\Gamma$ at $x$. 
Suppose $k=(k',k_d)\in \Z^d$, $\e k'\in Q'$ and let $P_{x^k}$ be the tangent plane to $\Gamma$ at $x^k=(\e k',g(\e k'))$. 
Then Theorem 2 
$\bf 1^\circ$ tells us that 
\[\cp_p(\Gamma_\e^k)=\cp_p(P_{x^k}\cap T_\e^k)+o(a_\e^{d-p}).\] 
If we set $P^k_\e = P+(-\e k',g(\e k'))$, then $P^k_\e$ will intersect the point $(\e k',g(\e k'))$. By assumption, $|\nu_x-\nu_{x^k}|\le \delta$, so 
\[\cp_p(P^k_\e\cap T_\e^k)=\cp_p(P_{x^k}\cap T_\e^k)+O(c_\delta a_\e^{d-p}),\] by Theorem 2 $\bf 2^\circ$. This gives 
$\cp_p(\Gamma_\e^k)=\cp_p(P^k_\e\cap T_\e^k)+O(c_\delta a_\e^{d-p})$. 
Since, by Theorem 1, the sequence $\{\e^{-1}g(\e k')\}_{k'\in\e^{-1}Q'}$ is uniformly distributed mod $1$ with discrepancy of order $\e^{1/3}$, the rescaled planes $\e^{-1}P^k_\e$ have the same distribution mod $1$, i.e.\ they are translates of $P$ and the translates have the same distribution. 
Using the proof of Lemma 4 of \cite{KS}, we conclude that 
\[
\lim_{\e\to0}\sum_{k\in \Z^n:k'\in \e^{-1}Q'}\cp_p(\{P_\e^k \}\cap T_\e^k) = \cp_{p,\nu_x}(T)\mathcal{H}^{d-1}(P_{Q'}), 
\]
where $P_{Q'}= \{x\in P:x'\in Q'\}$. 
Since we know that $\int_{B_\e^k}|\nabla w_\e^k|^pdx= \cp_p(\Gamma_\e^k)+o(a_\e^{d-p})$, the result follows from the fact that 
$\mathcal{H}^{d-1}(\Gamma_{Q'})=(1+O(c_\delta))\mathcal{H}^{d-1}(P_{Q'})$. 
\end{proof}

\medskip 

\begin{lemma}\label{cubelemma}
\[\lim_{\e\to0}\int_Q|\nabla w_\e|^pdx = \int_{\Gamma\cap Q} \cp_{p,\nu_x}(T)d\mathcal{H}^{d-1}.\]
\end{lemma}

\begin{proof}
The claim follows by decomposing the set $\{x'\in\R^{d-1}:(x',g(x'))\in \Gamma\cap Q\}$ into disjoint cubes $\{Q_j'\}$ that satisfy the hypothesis of  Lemma \ref{lemma3}. Since $\Gamma$ is $C^2$, we can find a finite number of disjoint cubes $\{Q_j\}_{j=1}^{N(\delta)}$, such that $\mathcal{H}^{d-1}(\Gamma\cap Q\setminus \cup_j Q_j\cap\Gamma)=0$ and $Q_j'$ is as in Lemma \ref{lemma3}. For all $x\in\Gamma\cap Q_j$ we have $x=(x',g(x))$ for $x'\in Q_j'$, after interchanging coordinate axes if necessary.  
Thus 
\begin{align*}
&\lim_{\e\to0}\int_Q|\nabla w_\e|^pdx = \sum_j\lim_{\e\to0}\sum_{k\in \Z^n:k'\in \e^{-1}Q'_j}\int_{B_\e^k}|\nabla w_\e^k|^pdx\\
&=\sum_{x^j\in Q_j'}[\cp_{p,\nu_{x^j}}(T)+ O(C_\delta)]\mathcal{H}^{d-1}(\Gamma_{Q_j'}) \\
&=\int_{\Gamma\cap Q}\cp_{p,\nu(x)}(T)d\mathcal{H}^{d-1}+O(C_\delta), 
\end{align*}
where in the last step we used that $\cp_{p,\nu(x)}(T)=\cp_{p,\nu_{x^j}}(T)+O(C_\delta)$ for all $x\in \Gamma_{Q_j'}$, by Lemma \ref{lemma3}. 
Sending $\delta \to0$ proves the lemma. 
\end{proof}

Having established Lemma \ref{cubelemma}, the rest of the proof of $\bf{H_2}$ and $\bf{H_3}$ is carried out precisely as in \cite{KS}, with Lemma \ref{cubelemma} above replacing Lemma 4 in \cite{KS}. The proof of Theorem \ref{mainthm} from $\bf{H_1}-\bf{H_3}$ is given in section 4 of \cite{KS} when $\Gamma$ is a hyper plane, and remains the same for the present case when $\Gamma$ is a convex surface.

\section*{Acknowledgements }
Part of the work was carried out at the Mittag-Leffler Institute during the programme in homogenisation and randomness, Winter 2014.

\end{document}